\documentclass[envcountsect,envcountsame,runningheads]{llncs}
\pdfoutput=1

\usepackage[utf8]{inputenc}
\usepackage[english]{babel}
\usepackage{amssymb,amsmath,mathtools,proof}
\usepackage[protrusion=true,expansion=true]{microtype}

\newcommand{\aaa}{\mathfrak{a}}
\newcommand{\mmm}{\mathfrak{m}}
\newcommand{\nnn}{\mathfrak{n}}
\newcommand{\ppp}{\mathfrak{p}}
\newcommand{\NN}{\mathbb{N}}
\newcommand{\ZZ}{\mathbb{Z}}
\newcommand{\defeq}{\vcentcolon=}
\renewcommand{\_}{\mathpunct{.}\,}

\begin{document}

\title{Maximal ideals in countable rings, constructively}
\author{Ingo Blechschmidt\inst{1} \and Peter Schuster\inst{2}}
\institute{Universität Augsburg, Universitätsstr. 14, 86159 Augsburg, Germany
\email{ingo.blechschmidt@math.uni-augsburg.de}
\and Università di Verona, Strada le Grazie 15, 37134 Verona, Italy
\email{petermichael.schuster@univr.it}}

\maketitle

\begin{abstract}
  The existence of a maximal ideal in a general nontrivial commutative
  ring is tied together with the axiom of choice.
  Following Berardi, Valentini and thus Krivine but using the relative interpretation of
  negation (that is, as ``implies~$0=1$'') we show, in constructive set theory with
  minimal logic, how for countable rings one can do without any kind of choice
  and without the usual decidability assumption that the ring is strongly
  discrete (membership in finitely generated ideals is decidable).
  By a functional
  recursive definition we obtain a maximal ideal in the sense that the quotient ring is a residue field
  (every noninvertible element is zero), and with strong discreteness
  even a geometric field (every element is either invertible or
  else zero). Krull's lemma for the related notion of prime ideal follows by
  passing to rings of fractions. All this equally applies to rings indexed by any well-founded set, and can be carried over to Heyting
  arithmetic with minimal logic.
  We further show how a metatheorem of Joyal and Tierney can be used to expand our
  treatment to arbitrary rings. Along the way we do a case study for proofs in
  algebra with minimal logic. An Agda formalization is available at an
  accompanying repository.\footnote{\url{https:/$\!$/github.com/iblech/constructive-maximal-ideals/}}
\end{abstract}

\noindent
Let~$A$ be a commutative ring with unit. The standard way of constructing a
maximal ideal of~$A$ is to apply Zorn's lemma to the set of proper ideals
of~$A$; but this method is less an actual construction and more an appeal to
the transfinite.

If~$A$ is countable with enumeration~$x_0,x_1,\ldots$, we can hope to provide a
more explicit construction by successively adding generators to the zero ideal,
skipping those which would render it improper:
\begin{align*}
  \mmm_0 &= \{ 0 \} &
  \mmm_{n+1} &= \begin{cases}
    \mmm_n + (x_n), & \text{if $1 \not\in \mmm_n + (x_n)$}, \\
    \mmm_n, & \text{else.}
  \end{cases}
\end{align*}
A maximal ideal is then obtained in the limit as the union of the intermediate
stages~$\mmm_n$. For instance, Krull in his 1929 Annals contribution~\cite[Hilfs\-satz]{krull:ohne} and books on constructive
algebra~\cite[Lemma~VI.3.2]{mines-richman-ruitenburg:constructive-algebra},~\cite[comment after Theorem~VII.5.2]{lombardi-quitte:constructive-algebra} proceed in this fashion.
A similar
construction concocts Henkin models for the purpose of proving
Gödel's completeness theorem for countable languages, successively adding
formulas which do not render the current set
inconsistent~\cite[Satz~I.56]{tarski:fundamental},~\cite[Lemma~1.5.7]{dalen:logic},~\cite[Lemma~III.5.4]{simpson:subsystems},~\cite[Lemma~2.1]{ishihara-khoussainov-nerode:decidable-kripke-models}.

This procedure avoids any form of choice by virtue of being a
functional recursive definition, but still requires some form of omniscience in
order to carry out the case distinction.
In the present text we study a variant of this construction, due to Berardi and
Valentini~\cite{berardi-valentini:krivine}, which
avoids any non-constructive principles and decidability assumptions, similar to
a construction which has been studied by
Krivine~\cite[p.~410]{krivine:completeness} and later Herbelin and
Ilik~\cite[p.~11]{herbelin-ilik:henkin} in the
context of Gödel's completeness theorem. In this generality, the resulting
maximal ideal has an elusive quality to it, but useful properties can still be
extracted; and not only do we recover the original construction under certain
decidability assumptions, we can also exploit a relativity phenomenon of
mathematical logic in order to drop, with some caveats, the assumption that~$A$ is countable.

\paragraph{Conventions.} Throughout this note, we fix a ring~$A$, and work in a constructive metatheory.
In the spirit of Lombardi and Quitté~\cite{lombardi-quitte:constructive-algebra}, we employ \emph{minimal
logic}~\cite{johansson:minimal}, where by ``not~$\varphi$'' we mean~``$\varphi \Rightarrow 1 =_A 0$'', and do \emph{not} assume any form of the axiom of choice. Consequently,
by~``$x \not\in M$'' we mean~$x \in M \Rightarrow 1 =_A 0$, and a subset~$M
\subseteq A$ is \emph{detachable} if and only if for all~$x \in A$, either~$x \in M$
or~$x \not\in M$. For general background on constructive mathematics, we refer
to~\cite{bauer:five-stages,bauer:int-mathematics,sep:constructive-mathematics}.

For an arbitrary subset~$M \subseteq A$, not necessarily
detachable, the ideal~$(M)$ generated by~$M$ is given by~$\bigl\{ \sum_{i=1}^n
a_i v_i \,\Big|\, n \geq 0, a_1,\ldots,a_n \in A, v_1,\ldots,v_n \in M \bigr\}$.
Notice that, for every element~$v \in (M)$, either $v = 0$ or $M$ is inhabited, depending on whether
$n=0$ or $n>0$ in $\sum_{i=1}^n a_i v_i$. This can also be seen from the alternative inductive
generation of $(M)$ by the following rules:\par
{\vspace*{-0.8em}\small\[
\infer{v \in (M)}{v=0}
\qquad\qquad
\infer{v \in (M)}{v \in M}
\qquad\qquad
\infer{v+w \in (M)}{v \in (M)\quad w \in (M)}
\qquad\qquad
\infer{av \in (M)}{\quad a\in A \quad v \in (M)}
\]}%
\noindent Here we adhere to the paradigm of generalized inductive definitions
\cite{aczel-rathjen:notes,aczel-rathjen:cstdraft,rathjen:genind}.

\section{A construction}
\label{sect:constr}

We assume that the ring~$A$ is countable, with~$x_0, x_1, \ldots$ an
enumeration of the elements of~$A$. We do \emph{not} assume that~$A$ is
discrete (that is, that~$x = y$ or~$x \neq y$ for all elements of~$A$) or that
it is strongly discrete (that is, that finitely generated
ideals of~$A$ are detachable). Up to Corollary~\ref{cor:is-prime-max}(a) below
we follow~\cite{berardi-valentini:krivine}.

We study the following recursive construction of ideals~$\mmm_0, \mmm_1,
\ldots$ of~$A$:
\begin{align*}
  \mmm_0 &\defeq \{0\} &
  \mmm_{n+1} &\defeq \mmm_n + (\{ x_n \,|\, 1 \not\in \mmm_n + (x_n) \}).
\end{align*}
Finally, we set~$\mmm \defeq \bigcup_n \mmm_n$. The construction
of~$\mmm_{n+1}$ from~$\mmm_n$ is uniquely specified, requiring no choices of
any form.

The set~$M_n \defeq \{ x_n \,|\, 1 \not\in \mmm_n + (x_n) \}$ occurring in this
construction contains the element~$x_n$ if and only if~$1 \not\in \mmm_n +
(x_n)$; it is obtained from the singleton set~$\{x_n\}$ by bounded separation.
This set~$M_n$ is inhabited precisely if~$1 \not\in \mmm_n + (x_n)$, in which case~$\mmm_{n+1} = \mmm_n + (x_n)$.
However, in the generality we work in, we cannot assume that~$M_n$ is empty or
inhabited.

We can avoid the case distinction by the flexibility of nondetachable
subsets, rendering it somewhat curious that---despite the conveyed flavor of a
conjuring trick---the construction can still be used to obtain concrete
positive results.

The ideal~$(M_n)$ is given by
$(M_n) = \{ a x_n \,|\, (a = 0) \vee (1 \not\in \mmm_n + (x_n)) \}$.

\begin{lemma}\label{lemma:omnibus}
\begin{enumerate}
\item[\textnormal{(a)}] The subset~$\mmm$ is an ideal.
\item[\textnormal{(b)}] The ideal~$\mmm$ is \emph{proper} in the sense that~$1 \not\in \mmm$.
\item[\textnormal{(c)}] For every number~$n \in \NN$, the following are equivalent: \\
(1)~$x_n \in \mmm_{n+1}$.\quad
(2)~$x_n \in \mmm$.\quad
(3)~$1 \not\in \mmm + (x_n)$.\quad
(4)~$1 \not\in \mmm_n + (x_n)$.
\end{enumerate}
\end{lemma}

\begin{proof}\begin{enumerate}
\item[(a)] Directed unions of ideals are ideals.
\item[(b)] Assume~$1 \in \mmm$. Then~$1 \in \mmm_n$ for some number~$n \geq 0$. We
verify~$1 = 0$ by induction over~$n$.
If~$n = 0$, then~$1 \in \mmm_0 = \{0\}$. Hence~$1 = 0$.

If~$n > 0$, then~$1 = y + a x_{n-1}$ for some elements~$a,y \in A$ such that~$y
\in \mmm_{n-1}$ and such that~$a = 0$ or~$1 \not\in \mmm_{n-1} + (x_{n-1})$.
In the first case, we have~$1 = y \in \mmm_{n-1}$, hence~$1 = 0$ by the induction
hypothesis. In the second case we have~$1 = 0$ by modus ponens applied to the
implication~$1 \not\in \mmm_{n-1} + (x_{n-1})$ and the fact~$1 \in \mmm_{n-1} +
(x_{n-1})$ (which follows directly from the equation~$1 = y + a x_{n-1}$).
\item[(c)] It is clear that~$(3) \Rightarrow (4) \Rightarrow (1) \Rightarrow
(2)$. It remains to show that~$(2) \Rightarrow (3)$.

Assume~$x_n \in \mmm$. In order to verify~$1 \not\in \mmm + (x_n)$,
assume~$1 \in \mmm + (x_n)$. Since~$\mmm + (x_n) \subseteq \mmm$,
we have~$1 \in \mmm$. Hence~$1 = 0$ by properness of~$\mmm$.
\end{enumerate}\end{proof}

\begin{corollary}\label{cor:is-prime-max}
\begin{enumerate}
\item[\textnormal{(a)}] The ideal~$\mmm$ is \emph{maximal} in the sense that it is proper and that for
all elements~$x \in A$, if $1 \not\in \mmm + (x)$, then~$x \in \mmm$.
\item[\textnormal{(b)}] The ideal~$\mmm$ is \emph{prime} in the
sense that it is proper and that for
all elements~$x,y \in A$, if~$xy \in \mmm$ and~$x \not\in \mmm$, then~$y \in
\mmm$.
\item[\textnormal{(c)}] The ideal~$\mmm$ is \emph{radical} in the sense that for every~$k \geq 0$, if~$x^k \in
\mmm$, then~$x \in \mmm$.
\end{enumerate}
\end{corollary}

\begin{proof}\begin{enumerate}
\item[(a)] Immediate by Lemma~\ref{lemma:omnibus}(c).
\item[(b)] This claim is true even for arbitrary maximal ideals: By maximality, it suffices to verify that~$1 \not\in
\mmm + (y)$. If~$1 \in \mmm + (y)$, then~$x = x\cdot1 \in (x) \cdot \mmm + (xy)
\subseteq \mmm$ by~$xy \in \mmm$, hence~$x \in \mmm$, thus~$1 = 0$ by~$x \not\in
\mmm$.
\item[(c)] Let~$x^k \in \mmm$. Then~$1 \not\in \mmm + (x)$, for if~$1 \in
\mmm + (x)$, then also~$1 = 1^k \in (\mmm + (x))^k \subseteq \mmm + (x^k)
\subseteq \mmm$. Hence~$x \in \mmm$ by maximality.
\end{enumerate}
\end{proof}

\begin{remark}The ideal~$\mmm$ is double negation stable: for every
ring element~$x$, if~$\neg\neg(x \in \mmm)$, then~$x \in \mmm$. This is because
by Lemma~\ref{lemma:omnibus}(c) membership of~$\mmm$ is a negative condition
and~$\neg\neg\neg\varphi \Rightarrow \neg\varphi$ is a tautology of minimal
logic.\end{remark}

This first-order maximality condition is  
equivalent~\cite{berardi-valentini:krivine} to the following higher-order version: For every ideal~$\nnn$ such that~$1
\not\in \nnn$, if~$\mmm \subseteq \nnn$, then~$\mmm = \nnn$.

The quotient ring~$A/\mmm$ is a \emph{residue field} in that~$1 \neq 0$
and that every element which is not invertible is zero---as with the real or
complex numbers in constructive mathematics.\footnote{Residue fields have
many of the basic properties of the fields from classical mathematics. For
instance, minimal generating families of vector spaces over residue fields are
linearly independent, finitely generated vector spaces do (up to~$\neg\neg$) have
a finite basis, monic polynomials possess splitting fields and Noether
normalization is available (the proofs
in~\cite{mines-richman-ruitenburg:constructive-algebra} can be suitably
adapted). The constructively rarer \emph{geometric fields}---those kinds of
fields for which every element is either invertible or zero---are required to
ensure, for instance, that kernels of matrices are finite dimensional and that bilinear forms
are diagonalizable.}
Each of the latter is in fact a \emph{Heyting field},
a residue field which also is a \emph{local ring}: if a finite
sum is invertible then one of the summands is.

\begin{example}If we enumerate~$\ZZ$ by~$0,1,-1,2,-2,\ldots$, the
ideal~$\mmm$ coincides with~$(2)$. If the enumeration starts with a
prime~$p$, the ideal~$\mmm$ coincides with~$(p)$.\end{example}

\begin{example}If~$A$ is a local ring with group of units~$A^\times$, then~$\mmm = A \setminus A^\times$.\end{example}

\begin{example}\label{ex:maximal-above}We can also use an
arbitrary ideal~$\aaa$ as~$\mmm_0$ instead of the zero ideal. All results in
this section remain valid once ``not~$\varphi$'' is redefined
as~``$\varphi \Rightarrow 1\in\aaa$''; the resulting ideal~$\mmm$ is then a
maximal ideal above~$\aaa$; it is proper in the sense that~$1 \in \mmm
\Rightarrow 1 \in \aaa$. It can also be obtained by applying the original
version of the construction in the quotient ring~$A/\aaa$ (which is again
countable) and taking the inverse image of the resulting ideal along the
canonical projection~$A \to A/\aaa$.\end{example}

\begin{example}Assume that~$A$ is a field. Let~$f \in A[X]$ be a nonconstant monic
polynomial. Since~$f$ is monic, it is not invertible; thus
Example~\ref{ex:maximal-above} shows that there is a maximal ideal~$\mmm$
above~$(f)$. Hence~$A[X]/\mmm$ is a field in which~$f$ has a zero, namely the equivalence class of~$X$.
Iterating this \emph{Kronecker construction}, we obtain a splitting field of~$f$. No
assumption regarding decidability of reducibility has to be made, but in return
the resulting fields are only residue fields.\end{example}

If we can decide whether a finitely generated ideal contains the
unit or not, we can improve on Corollary~\ref{cor:is-prime-max}(a). For instance this is the case for
strongly discrete rings such as the ring~$\ZZ$, more generally for the ring of
integers of every number field, and for polynomial rings over discrete
fields~\cite[Theorem~VIII.1.5]{mines-richman-ruitenburg:constructive-algebra}.

\begin{proposition}\label{prop:with-test}
Assume that for every finitely generated ideal~$\aaa \subseteq A$ we have~$1
\not\in \aaa$ or~$\neg(1 \not\in \aaa)$. Then:
\begin{enumerate}
\item[\textnormal{(a)}] Each ideal~$\mmm_n$ is finitely generated.
\item[\textnormal{(b)}] The ideal~$\mmm$ is detachable.
\end{enumerate}
If even~$1 \in \aaa$ or~$1 \not\in \aaa$ for every finitely generated ideal~$\aaa \subseteq
A$, then:
\begin{enumerate}
\addtocounter{enumi}{2}
\item[\textnormal{(c)}] The ideal~$\mmm$ is maximal in the strong sense that for every element~$x
\in A$,~$x \in \mmm$ or~$1 \in \mmm + (x)$, which is to say that the quotient ring~$A/\mmm$ is a
\emph{geometric field} (every element is zero or invertible).\footnote{This
notion of a maximal ideal, together with the corresponding one of a complete
theory in propositional logic, has been generalized to the concept of a
complete coalition~\cite{schuster-wessel:cie2020,schuster-wessel:jacincpred} for an abstract inconsistency
predicate.}
\end{enumerate}
\end{proposition}

\begin{proof}We verify claim~(a) by induction over~$n$. The case~$n = 0$ is
clear. Let~$n > 0$. By the induction hypothesis, the ideal~$\mmm_{n-1}$ is finitely
generated, hence so is~$\mmm_{n-1} + (x_{n-1})$. By assumption,~$1 \not\in \mmm_{n-1} +
(x_{n-1})$ or~$\neg(1 \not\in \mmm_{n-1} + (x_{n-1}))$. In the first
case~$\mmm_n = \mmm_{n-1} + (x_{n-1})$. In the second case~$\mmm_n =
\mmm_{n-1}$. In both cases
the ideal~$\mmm_n$ is finitely generated.

To verify claim~(b), let an element~$x_n \in A$ be given. By assumption,~$1
\not\in \mmm_n + (x_n)$ or~$\neg(1 \not\in \mmm_n + (x_n))$. Hence~$x_n \in
\mmm$ or~$x_n \not\in \mmm$ by Lemma~\ref{lemma:omnibus}(c).

For claim~(c), let an element~$x_n \in A$ be given. If~$1 \in \mmm_n + (x_n)$,
then also~$1 \in \mmm + (x_n)$. If~$1 \not\in \mmm_n + (x_n)$, then~$x_n \in
\mmm$ by Lemma~\ref{lemma:omnibus}(c).
\end{proof}

Remarkably, under the assumption of Proposition~\ref{prop:with-test}, the ideal~$\mmm$ is detachable even though in
general it fails to be finitely generated. Usually in constructive mathematics, ideals which are not
finitely generated are seldom detachable. For instance the ideal~$\{ x \in
\ZZ \,|\, x = 0 \vee \varphi \} \subseteq \ZZ$ is detachable if and only
if~$\varphi \vee \neg\varphi$.

\begin{remark}\label{rem:via-generators}There is an equivalent description of the
maximal ideal~$\mmm$ which uses sets~$G_n$ of generators as proxies for the
intermediate ideals~$\mmm_n$:
\begin{align*}
  G_0 &\defeq \emptyset &
  G_{n+1} &\defeq G_n \cup \{ x_n \,|\, 1 \not\in (G_n \cup \{ x_n \}) \}
\end{align*}
An induction establishes the relation~$(G_n) = \mmm_n$; setting~$G \defeq
\bigcup_{n\in\NN} G_n$, the analogue of Lemma~\ref{lemma:omnibus}(c) states
that for every number~$n \in \NN$, the following are equivalent:
(1)~$x_n \in G_{n+1}$.
(2)~$x_n \in G$.
(3)~$1 \not\in (G) + (x_n)$.
(4)~$1 \not\in (G_n) + (x_n)$.

In particular, not only do we have that~$(G) = \mmm$, but~$G$ itself is already
an ideal. This description of~$\mmm$ is in a sense more ``economical'' as the
intermediate stages~$G_n$ are smaller (not yet being ideals), enabling
arithmetization in Section~\ref{sect:arithmetization}.
\end{remark}

\newcommand{\rightrightharpoonup}{\mathrel{\mathrlap{\rightharpoonup}\mkern1mu\rightharpoonup}}
\begin{remark}All results in this section carry over mutatis mutandis if~$A$ is
only assumed to be subcountable, that is, if we are only
given a \emph{partially defined} surjection~$\NN \rightrightharpoonup A$. In
this case, we are given an enumeration~$x_0,x_1,\ldots$ where some~$x_i$
might not be defined; we then define~$\mmm_{n+1} \defeq
\mmm_n + (\{ x_n \,|\, \text{$x_n$ is defined} \wedge 1 \not\in \mmm_n + (x_n) \})$.
The generalization to the subcountable case is particularly useful in the
Russian tradition of constructive mathematics as exhibited by the ef{}fective
topos~\cite{hyland:effective-topos,oosten:realizability,phoa:effective,bauer:c2c},
where many rings of interest are subcountable, including uncountable ones such as the real
numbers~\cite[Prop.~7.2]{hyland:effective-topos}.
\end{remark}

\section{On the intersection of all prime ideals}

Classically, Krull's lemma states that the intersection of all prime ideals is the
\emph{nilradical}, the ideal~$\sqrt{(0)}$ of all nilpotent elements. In our
setup, we have the following substitute concerning complements:
\[ \sqrt{(0)}^c =
  \bigcup_{\substack{\text{$\ppp \subseteq A$}\\\text{$\ppp$ prime}\\\text{$\ppp$ $\neg\neg$-stable}}} \ppp^c =
  \bigcup_{\substack{\text{$\ppp \subseteq A$}\\\text{$\ppp$ prime}\\\text{$\ppp$ radical}}} \ppp^c. \]
\newpage

\begin{lemma}\label{lemma:x-prime}
Let~$x \in A$. Then there is an ideal~$\ppp \subseteq A$ which is
\begin{enumerate}
\item ``$x$-prime'' in the sense that
$1 \in \ppp \Rightarrow x \in \sqrt{(0)}$ and
$ab \in \ppp \wedge \bigl(b \in \ppp \Rightarrow x \in \sqrt{(0)}\bigr) \Longrightarrow
   a \in \ppp$,
that is, prime if the negations occurring in the definition of ``prime ideal''
are understood as~``$\varphi \Rightarrow x \in \sqrt{(0)}$'',
\item ``$x$-stable'' in the sense that
$\bigl((a \in \ppp \Rightarrow x \in \sqrt{(0)}) \Rightarrow x \in \sqrt{(0)}\bigr)
  \Rightarrow a \in \ppp$,
\item radical,
\item and such that~$x \in \ppp$ if and only if~$x$ is nilpotent.
\end{enumerate}
\end{lemma}

\begin{proof}The localization~$A[x^{-1}]$ is again countable, hence the
construction of Section~\ref{sect:constr} can be carried out to obtain a
maximal (and hence prime) ideal~$\mmm \subseteq A[x^{-1}]$. Every negation
occurring in the terms ``maximal ideal'' and ``prime ideal'' refers to~$1 = 0$
in~$A[x^{-1}]$, which is equivalent to~$x$ being nilpotent.

The preimage of~$\mmm$ under the localization homomorphism~$A \to A[x^{-1}]$ is
the desired~$x$-prime ideal.
\end{proof}

\begin{corollary}[Krull~\cite{krull:ohne}]\label{cor:nilp-prime}Let~$x \in A$ be an element which is not nilpotent. Then there is a
(radical and $\neg\neg$-stable) prime ideal~$\ppp \subseteq A$ such that~$x \not\in \ppp$.
\end{corollary}

\begin{proof}Because~$x$ is not nilpotent, the notion of an~$x$-prime ideal and
an ordinary prime ideal coincide. Hence the claim follows from
Lemma~\ref{lemma:x-prime}.\end{proof}

An important part of constructive algebra is to devise tools to import
proofs from classical commutative algebra into the constructive
setting.\footnote{Forms of Zorn's Lemma similar to Krull's Lemma
feature prominently in algebra; to wit, in ordered algebra there are the
Artin--Schreier theorem for fields, Levi's theorem for Abelian groups and
Ribenboim's extension to modules. Dynamical algebra aside, to which we will
come back later, these statements have recently gained attention from the angle
of proof theory at large; see, for example,
\cite{rin:ukl,rin:edde,rin:edde:full,wessel:ordering,schuster-wessel:ext,wessel:conred,bon:rib,pow:occ}.}
The following two statements are established test cases exploring the power of
such tools~\cite{schuster:induction,schuster:inductionjournal,persson:constructive-spectrum,powell-schuster-wiesnet:krull,swy:dynprime,schuster-wessel:indeterminacy,banaschweski-vermeulen:radical,richman:trivial-rings,coquand-lombardi:logical,coquand-lombardi-roy:dynamicalmethod}.

\begin{proposition}\label{prop:test-cases}Let~$f \in A[X]$ be a polynomial.
\begin{enumerate}
\item If~$f$ is nilpotent in~$A[X]$, then all coefficients of~$f$ are nilpotent in~$A$.
\item If~$f$ is invertible in~$A[X]$, then all nonconstant coefficients of~$f$ are nilpotent.
\end{enumerate}
\end{proposition}

These facts have abstract classical proofs employing Krull's lemma as follows.

\begin{quote}\small
  \textbf{Proof of 1.} Simple induction if~$A$ is
  reduced; the general case reduces to this one: For every prime ideal~$\ppp$,
  the coefficients of~$f$ vanish over the reduced ring~$A/\ppp$. Hence they are
  contained in all prime ideals and are thereby
  nilpotent.\medskip

  \textbf{Proof of 2.} Simple induction if~$A$ is an
  integral domain; the general case reduces to this one: For every prime
  ideal~$\ppp$, the nonconstant coefficients of~$f$ vanish over the integral
  domain~$A/\ppp$. Hence they are contained in all prime ideals
  and are thereby nilpotent.
\end{quote}

Both statements admit direct computational proofs which do not
refer to prime ideals; the challenge is not to find such proofs, but rather to
imitate the two classical proofs above constructively, staying as close as
possible to the original. It is remarkable that the construction of
Section~\ref{sect:constr} meets this challenge at all, outlined as follows, despite its fundamental
reliance on nondetachable subsets.

We continue assuming that~$A$ is countable:
Section~\ref{sect:wlog} indicates how this assumption can be dropped in quite
general situations, while for the purposes of specific challenges such as Proposition~\ref{prop:test-cases} we could also
simply pass to the countable subring generated by the polynomial coefficients
or employ the method of indeterminate coefficients.

\begin{proof}[of Proposition~\ref{prop:test-cases}]
The first claim follows from a simple induction if~$A$ is a reduced
ring.
In the general case, write~$f = a_n X^n + a_{n-1} X^{n-1} + \cdots + a_0$. Let~$\ppp$
be a radical~$a_n$-prime ideal as in Lemma~\ref{lemma:x-prime}. Since~$A/\ppp$
is reduced, the nilpotent coefficient~$a_n$ vanishes over~$A/\ppp$. Thus~$a_n \in \ppp$,
hence~$a_n$ is nilpotent. Since the polynomial~$f - a_n X^n$ is again
nilpotent, we can continue by induction.

The second claim follows by a simple inductive argument if~$A$ is an
integral domain with double negation stable equality.
In the general case, write~$f = a_n X^n + \cdots + a_0$
and assume~$n \geq 1$. To reduce to the integral situation, let~$\ppp$ be
an~$a_n$-prime ideal as in Lemma~\ref{lemma:x-prime}.
With negation~``$\neg\varphi$'' understood as~``$\varphi \Rightarrow a_n \in
\sqrt{(0)}$'', the quotient ring~$A/\ppp$ is an integral domain with double
negation stable equality.
Hence~$a_n = 0$ in~$A/\ppp$, so~$a_n \in \ppp$ whereby~$a_n$ is nilpotent. The
polynomial~$f - a_n X^n$ is again invertible in~$A[X]$ (since the group of
units is closed under adding nilpotent elements) so that we can continue by
induction.
\end{proof}

Just as Corollary~\ref{cor:nilp-prime} is a constructive substitute
for the recognition of the intersection of all prime ideals as the nilradical,
the following proposition is a substitute for the classical fact that
the intersection of all maximal ideals is the Jacobson radical.
As is customary in constructive
algebra~\cite[Section~IX.1]{lombardi-quitte:constructive-algebra}, by
\emph{Jacobson radical} we mean the ideal
$\{ x \in A \,|\, \forall y \in A\_ 1 - xy \in A^\times \}$.

\begin{proposition}Let~$x \in A$. If~$x$ is \emph{apart} from the Jacobson radical (that is, $1-xy \not\in A^\times$ for some element~$y$), then
there is a maximal ideal~$\mmm$ such that~$x \not\in \mmm$.
\end{proposition}

\begin{proof}The standard proof as
in~\cite[Lemma~IX.1.1]{lombardi-quitte:constructive-algebra} applies: There is
an element~$y$ such that~$1-xy$
is not invertible. By Example~\ref{ex:maximal-above}, there is an ideal~$\mmm$
above~$\aaa \defeq (1-xy)$ which is maximal not only as an ideal of~$A/\aaa$
(where~``$\neg\varphi$'' means~``$\varphi \Rightarrow 1 \in \aaa$'') but also as an
ideal of~$A$ (where~``$\neg\varphi$'' means~``$\varphi \Rightarrow 1 = 0$''). If~$x
\in \mmm$, then~$1 = (1-xy) + xy \in \mmm$; hence~$x \not\in \mmm$.
\end{proof}

The two test cases presented in Proposition~\ref{prop:test-cases} only concern
prime ideals. In contrast, the following example crucially rests on the
maximality of the ideal~$\mmm$.

\begin{proposition}Let~$M \in A^{n \times m}$ be a matrix with more rows than
columns. Assume that the induced linear map~$A^m \to A^n$ is surjective.
Then~$1 = 0$.
\end{proposition}

\begin{proof}By passing to the quotient~$A/\mmm$, we may assume that~$A$ is a
residue field. In this case the claim is standard linear algebra:
If any of the matrix entries
is invertible, the matrix could be transformed by elementary row and
column operations to a matrix of the form~$\left(\begin{smallmatrix}1 & 0 \\ 0 &
M'\end{smallmatrix}\right)$, where the induced linear map of the submatrix~$M'$ is again
surjective. Thus~$1 = 0$ by induction.

Hence by the residue field property all matrix entries are zero.
But the vector $(1,0,\ldots,0)\in A^n$ still
belongs to the range of $M=0$, hence~$1=0$ by~$n > 0$.
\end{proof}

\begin{remark}\label{rem:suslin}
A more significant case study is Suslin's lemma, the fundamental and originally
non-constructive ingredient in his second solution of Serre's
problem~\cite{suslin:structure}. The classical proof, concisely recalled in
Yengui's constructive account~\cite{yengui:maximal}, reduces modulo
maximal ideals. The construction of Section~\ref{sect:constr} offers a
constructive substitute. However, since gcd computations
are required in the quotient rings, it is not enough that they are residue
fields; they need to be geometric fields. Hence our approach has to be combined
with the technique variously known as \emph{Friedman's trick}, \emph{nontrivial
exit continuation} or \emph{baby version of Barr's theorem} in order to yield a
constructive
proof~\cite{friedman:trick,murthy:classical-proofs,barr:without-points,blechschmidt:generalized-spaces}.
\end{remark}

\section{In Heyting arithmetic}
\label{sect:arithmetization}

The construction presented in Section~\ref{sect:constr} crucially rests on the
flexibility of nondetachable subsets: In absence of additional
assumptions as in Proposition~\ref{prop:with-test},
we cannot give the ideals~$\mmm_n$ by decidable
predicates~$A \to \{0,1\}$---without additional hypotheses on~$A$, membership of the
ideals~$\mmm_n$ is not decidable. As such, the construction is
naturally formalized in intuitionistic set theories such as~\textsc{czf}
or~\textsc{izf}, which natively support such flexible subsets.

In this section, we explain how with some more care, the construction can also
be carried out in much weaker foundations such as Heyting
arithmetic~\textsc{ha}. While formulation in classical Peano arithmetic~\textsc{pa}
is routine, the development in~\textsc{ha} crucially rests on a specific
feature of the construction, namely that the condition for membership is a
negative condition.

To set the stage, we specify what we mean by a \emph{ring} in the context of
arithmetic. One option would be to decree that an arithmetized ring should be a single
natural number coding a finite set of ring elements and the graphs of the
corresponding ring operations; however, this perspective is too narrow, as we
also want to work with infinite rings.

Instead, an arithmetized ring should be given by a ``formulaic setoid with ring
structure'', that is: by a formula~$A(n)$ with free
variable~$n$, singling out which natural numbers constitute
representatives of the ring elements; by a formula~$E(n,m)$ describing which
representatives are deemed equivalent; by a formula~$Z(n)$ singling out
representatives of the zero element; by a formula~$P(n,m,s)$ singling out
representatives~$s$ of sums; and so on with the remaining data constituting
a ring; such that axioms such as\par
{\vspace*{-1.2em}\small\begin{align*}
  \forall n\_ & Z(n) \Rightarrow A(n) && \text{``every zero representative belongs to the ring''} \\
  \exists n\_ & Z(n) && \text{``there is a zero representative''} \\
  \forall n,m\_ & Z(n) \wedge Z(m) \Longrightarrow E(n,m) && \text{``every two zero representatives are equivalent''} \\
  \forall z,n\_ & Z(z) \wedge A(n) \Longrightarrow P(z,n,n) && \text{``zero is neutral with respect to addition''}
\end{align*}}%
hold. This conception of arithmetized rings deviates from the usual definition
in reverse mathematics~\cite[Definition~III.5.1]{simpson:subsystems} to support
quotients even when~\textsc{ha} cannot verify the existence of
canonical representatives of equivalence classes.

Although first-order arithmetic cannot quantify over ideals of arithmetized
rings, specific ideals can be given by formulas~$I(n)$ such that axioms such as
{\small\begin{align*}
  \forall n\_ & I(n) \Rightarrow A(n) && \text{``$I \subseteq A$''} \\
  \exists n\_ & Z(n) \wedge I(n) && \text{``$0 \in I$''}
\end{align*}}%
hold. It is in this sense that we are striving to adapt the construction of
Section~\ref{sect:constr} to describe a maximal ideal.

In this context, we can arithmetically imitate any set-theoretic description of
a single ideal as a subset cut out by an explicit first-order formula. However, for
recursively defined families of ideals, we require a suitable recursion
theorem: If we are given (individual formulas~$M_n(x)$ indexed by numerals
representing) ideals~$\mmm_0,\mmm_1,\mmm_2,\ldots$, we cannot generally
form~$\bigcup_{n\in\NN} \mmm_n$, as the naive formula~``$\bigvee_{n\in\NN} M_n(x)$''
representing their union would have infinite length. We can take the union only
if the family is \emph{uniformly represented} by a single formula~$M(n,x)$ (expressing
that~$x$ represents an element of~$\mmm_n$).

This restriction is a blocking issue for arithmetizing the
construction of the chain~$\mmm_0 \subseteq \mmm_1 \subseteq \cdots$
of Section~\ref{sect:constr}. Because~$\mmm_n$ occurs in the
definition of~$\mmm_{n+1}$ in negative position, naive arithmetization
results in formulas of unbounded logical complexity, suggesting that
a uniform definition might not be possible.

This issue has a counterpart in type-theoretic foundations of mathematics,
where the family~$(\mmm_n)_{n \in \NN}$ cannot be given as an inductive family
(failing the positivity check), and is also noted, though not resolved, in
related work~\cite[p.~11]{herbelin-ilik:henkin}.
The issue does not arise in the context of~\textsc{pa}, where the law of
excluded middle allows us to bound the logical complexity: We can blithely define
the joint indicator function~$g(n,i)$ for the sets~$G_n$ (such that~$G_n = \{
x_i \,|\, i \in \NN, g(n,i) = 1 \}$) of Remark~\ref{rem:via-generators} by the recursion
\begin{align*}
  g(0,i) &= 0 \\
  g(n+1,i) &= \begin{cases}
    1, & \text{if $g(n,i) = 1 \vee (i = n \wedge 1 \not\in
    (g(n,0)x_0,\ldots,g(n,n-1)x_{n-1},x_n))$} \\
    0, & \text{else.}
  \end{cases}
\end{align*}
This recursion can be carried out within~\textsc{pa} since the recursive step
only references the finitely many values~$g(n,0),\ldots,g(n,i)$.
Heyting arithmetic, however, does not support this case distinction. The
formalization of the construction in~\textsc{ha} is only unlocked by the following
direct characterization.

\begin{lemma}\label{lemma:uniform-char}(In the situation of Remark~\ref{rem:via-generators}.)
For every finite binary sequence~$v = [v_0,\ldots,v_{n-1}]$, set~$\aaa_v
\defeq (v_0x_0,\ldots,v_{n-1}x_{n-1},x_n)$. Then:
\begin{enumerate}
\item For every such sequence~$v = [v_0,\ldots,v_{n-1}]$, if
$\bigwedge_{i=0}^{n-1} (v_i = 1 \Leftrightarrow 1 \not\in \aaa_{[v_0,\ldots,v_{i-1}]})$,
then~$\aaa_v = (G_n) + (x_n)$.
In particular, in this case~$x_n \in G$ if and only if~$1 \not\in \aaa_v$.
\item For every natural number~$n \in \NN$,
\vspace*{-1.2em}
\[ x_n \in G \quad\Longleftrightarrow\quad \neg
  \exists v \in \{0,1\}^n\_
    1 \in \aaa_v \wedge
      \bigwedge_{i=0}^{n-1} (v_i = 1 \Leftrightarrow 1 \not\in \aaa_{[v_0,\ldots,v_{i-1}]}). \]
\end{enumerate}
\end{lemma}

\begin{proof}The first part is by induction, employing the equivalences of
Remark~\ref{rem:via-generators}. The second rests on the tautology
$\neg\alpha \Longleftrightarrow \neg(\alpha \wedge (\varphi \vee \neg\varphi))$:

\vspace*{-1.2em}\small
\begin{align*}
  x_n \in G &
  \Longleftrightarrow \neg\bigl(1 \in (G_n) + (x_n)\bigr)
  \Longleftrightarrow \neg\bigl(1 \in (G_n) + (x_n) \ \wedge\ \bigwedge_{i=0}^{n-1} (x_i \in G \vee x_i \not\in G)\bigr) \\
  & \Longleftrightarrow \neg\exists v \in \{0,1\}^n\_
    \Bigl(1 \in (G_n) + (x_n) \ \wedge\ \bigwedge_{i=0}^{n-1} (v_i = 1
    \Leftrightarrow x_i \in G)\Bigr) \\
  & \Longleftrightarrow \neg\exists v \in \{0,1\}^n\_
    \Bigl(1 \in \aaa_v \ \wedge\ \bigwedge_{i=0}^{n-1} (v_i = 1
    \Leftrightarrow 1 \not\in \aaa_{[v_0,\ldots,v_{i-1}]})\Bigr) 
\end{align*}
\end{proof}

\noindent
Condition~(2) is manifestly formalizable in
arithmetic, uniformly in~$n$.

\section{For general rings}
\label{sect:wlog}

The construction in Section~\ref{sect:constr} of a maximal ideal applies to
countable rings. In absence of the axiom of choice, some restriction on the
rings is required, as it is well-known that the statement that any nontrivial
ring has a maximal ideal implies (over Zermelo--Fraenkel set
theory~\textsc{zf}) the axiom of choice~\cite{scott:prime-ideals,hodges:krull,banaschewski:krull,erne:krull,howard-rubin:ac}.

However, this limitation only pertains to the abstract existence of maximal
ideals, not to concrete consequences of their existence. Mathematical
logic teaches us by way of diverse examples to not conflate these two concerns. For
instance, although~\textsc{zf} does not prove the axiom of choice, it does
prove every theorem of~\textsc{zfc} pertaining only to natural numbers (by
interpreting a given~\textsc{zfc}-proof in the constructible universe~$L$
and exploiting that the natural numbers are absolute between~$V$
and~$L$~\cite{goedel:ac-gch,schoenfield:predicativity}); similarly, although intuitionistic Zermelo--Fraenkel set
theory~\textsc{izf} does not prove the law of excluded middle, it does prove
every~$\Pi^0_2$-theorem of~\textsc{zf} (by the double negation translation
combined with Friedman's continuation trick~\cite{friedman:double-negation-translation}).
A similar phenomenon concerns countability, as follows.

\subsubsection{A metatheorem by Joyal and Tierney}

Set theory teaches us that
whether a given set is countable depends not only on the set itself, but is
more aptly regarded as a property of the ambient universe~\cite{hamkins:multiverse}: Given any
set~$M$, there is a (non-Boolean) extension of the universe in which~$M$ becomes countable.
Remarkably, the passage to such an extension preserves and reflects
first-order logic.
Hence we have the metatheorem that \emph{countability assumptions from
intuitionistic proofs of first-order statements can always be mechanically
eliminated.}\footnote{For every set~$M$, there is a certain locale~$X$ (the
\emph{classifying locale of enumerations of~$M$}) which is overt, positive and such
that its constant sheaf~$\underline{M}$ is countable in the sense of the
internal language of the topos of sheaves over~$X$. A given
intuitionistic proof can then be interpreted in this topos~\cite{caramello:preliminaries,maietti:modular,shulman:categorical-logic}; since the constant
sheaf functor preserves first-order logic (by overtness), the
sheaf~$\underline{M}$ inherits any first-order assumptions about~$M$ required
by the proof; and since it also reflects first-order logic (by overtness and
positivity), the proof's conclusion descends to~$M$.\par When we apply the
construction of Section~\ref{sect:constr} internally in this topos, the result
will be a certain sheaf of ideals; it is in that sense that every ring
constructively possesses a maximal ideal. This sheaf will not be constant,
hence not originate from an actual ideal of the given ring; but first-order
consequences of the existence of this sheaf of ideals pass down to the ring.
Details are provided by Joyal and
Tierney~\cite[pp.~36f.]{joyal-tierney:grothendieck}, and introductions to
pointfree topology and topos theory can be found
in~\cite{blechschmidt:generalized-spaces,johnstone:point,vickers:continuity,vickers:locales-toposes}. A predicative account on the basis of~\cite{maietti:au,vickers:sketches,crosilla:predicativity}
is also possible. The phenomenon that size is relative also emerges in the
Löwenheim--Skolem theorem.} Crucially, the first-order restriction is only on the form
of the statements, not on the form of the proofs. These may freely employ higher-order constructs.

``First-order'' statements are statements which only refer to elements, not to
subsets; for instance, the statements of Proposition~\ref{prop:test-cases} are
first-order and hence also hold without the countability assumption.
In contrast, the statement ``there is a maximal ideal'' is a higher-order
statement; hence we cannot eliminate countability assumptions from proofs of
this statement.

The metatheorem expands the applicability of the construction of
Section~\ref{sect:constr} and underscores the value of its intuitionistic
analysis---the metatheorem cannot be applied to eliminate countability assumptions from classical proofs.
Taken together, they strengthen the
view of maximal ideals as convenient fictions~\cite[Section~1]{schuster-wessel:krull}. Maximal ideals can
carry out their work by any of the following possibilities:
(1)~For countable (or well-founded) rings, no help is required.
Section~\ref{sect:constr} presents an explicit construction of a maximal ideal.
(2)~For arbitrary rings, the existence of a maximal ideal follows from the
axiom of choice.
(3)~Intuitionistic first-order consequences of the existence of a maximal
ideal are true even if no actual maximal ideal can be constructed.

\subsubsection{Comparison with dynamical algebra}

The dynamical approach~\cite[Section~XV.6]{lombardi-quitte:constructive-algebra},~\cite{coquand-lombardi-roy:dynamicalmethod},~\cite{yengui:constructive},~\cite{duval:about}
is another technique for constructively reinterpreting,
without countability assumptions, classical proofs involving maximal ideals.
We sketch here how the dynamical approach is intimately connected with the
technique of this section, even though it is cast in entirely different
language.

Suppose that a given classical proof appeals to the maximality condition ``$x \in \mmm$
or~$1 \in \mmm + (x)$'' (``$x$ is zero modulo~$\mmm$ or invertible
modulo~$\mmm$'') only for a finite number~$x_0,\ldots,x_{n-1}$ of ring elements
fixed beforehand. In this case we can, even if no enumeration of all elements
of~$A$ exists or is available, apply the construction in Section~\ref{sect:constr} to
this finite enumeration and use the resulting ideal~$\mmm_n$ as a partial
substitute for an intangible maximal ideal.

The tools from pointfree topology
driving Joyal and Tierney's metatheorem widen the applicability of this partial
substitute to cases where the inspected ring elements are
not fixed beforehand, by dynamically growing the partial enumeration as the
proof runs its course. If required, a continuation-passing style transform as
in Remark~\ref{rem:suslin} can upgrade the maximal ideal from only
satisfying~``$1 \not\in \mmm + (x)$ implies~$x \in
\mmm$'' to satisfying the stronger condition~``$x \in \mmm$ or~$1 \in \mmm
+ (x)$''.
Unfolding the construction of~$\mmm$ and the proof of Joyal
and Tierney's metatheorem, we arrive at the dynamical method.

\subsubsection{Acknowledgments}
The present study was carried out within the project ``Reducing complexity in
algebra, logic, combinatorics -- REDCOM'' belonging to the program ``Ricerca
Scientifica di Eccellenza 2018'' of the Fondazione Cariverona and GNSAGA of the INdAM.%
\footnote{The opinions expressed in this paper are solely those of the
authors.} Important steps towards this paper were made
during the Dagstuhl Seminar 21472 ``Geometric Logic, Constructivisation, and
Automated Theorem Proving'' in November 2021. This paper would not have come to
existence without the authors' numerous discussions with Daniel Wessel, and
greatly benefited from astute comments of Karim Becher, Nicolas Daans, Kathrin Gimmi, Matthias Hutzler,
Lukas Stoll and the three anonymous reviewers.

\bibliographystyle{splncs04}
\bibliography{main}

\end{document}